\documentclass[a4paper, reqno, 12pt]{amsart}

\usepackage[utf8]{inputenc}
\usepackage{amsthm,amsfonts,amssymb,amsmath,amsxtra,amsrefs}
\usepackage{mathtools}
\usepackage{bm}
\usepackage{latexsym}
\usepackage{stmaryrd}
\usepackage{graphicx}
\usepackage{tikz}
\usepackage{tikz-cd}
\usepackage{todonotes,cancel}
\usepackage[margin=1.25in]{geometry}
\usepackage{mathrsfs}
\usepackage{bm}
\usepackage[all]{xy}
\SelectTips{cm}{}
\usepackage{xr-hyper}
\usepackage[colorlinks=false,
   citecolor=Black,
   linkcolor=Red,
   urlcolor=Blue]{hyperref}
\usepackage{verbatim}
\RequirePackage{xspace}
\RequirePackage{etoolbox}
\RequirePackage{varwidth}
\RequirePackage{enumitem}
\RequirePackage{tensor}
\RequirePackage{mathtools}
\RequirePackage{longtable}
\RequirePackage{multirow}

\usepackage{rotating}

\tikzset{
    labl/.style={anchor=south, rotate=90, inner sep=.5mm}
}

\tikzset{>=latex}
\tikzstyle{vthick}=[line width=1.8pt]

\newcommand\drawpath[2]{%
  \foreach \too [count=\c from 1] in {#1}
  {
  \ifthenelse{\c=1}
  {\xdef\from{\too}}
  {\path (\from) edge [->, #2] (\too);
    \xdef\from{\too}}
  };
}

\newtheorem{thm}{Theorem}
\newtheorem{theorem}[thm]{Theorem}
\newtheorem{thmintro}{Theorem}

\newtheorem{prop}[thm]{Proposition}

\newtheorem{lem}[thm]{Lemma}
\newtheorem{lemma}[thm]{Lemma}

\newtheorem{cor}[thm]{Corollary}

\theoremstyle{definition}
\newtheorem{definition}[thm]{Definition}

\newtheorem{example}[thm]{Example}

\newtheorem{rem}[thm]{Remark}

\numberwithin{equation}{section}
\numberwithin{thm}{section}

\newcommand{\BC}{\ensuremath{\mathbb {C}}\xspace}

\newcommand{{\BG}}{\ensuremath{\mathbb {G}}\xspace}

\newcommand{{\BK}}{\ensuremath{\mathbb {K}}\xspace}

\newcommand{\BN}{\ensuremath{\mathbb {N}}\xspace}

\newcommand{\BQ}{\ensuremath{\mathbb {Q}}\xspace}
\newcommand{\BR}{\ensuremath{\mathbb {R}}\xspace}

\newcommand{\BZ}{\ensuremath{\mathbb {Z}}\xspace}

\newcommand{\CD}{\ensuremath{\mathcal {D}}\xspace}

\newcommand{\CL}{\ensuremath{\mathcal {L}}\xspace}

\newcommand{\CP}{\ensuremath{\mathcal {P}}\xspace}

\newcommand{\CU}{\ensuremath{\mathcal {U}}\xspace}

\newcommand{\x}{\times}
\newcommand{\ox}{\otimes}
\newcommand{\wt}{\widetilde}

\newcommand{\bfd}{\textnormal{\textbf{d}}}

\newcommand{\bfi}{\textnormal{\textbf{i}}}
\newcommand{\bfj}{\textnormal{\textbf{j}}}

\newcommand{\bfs}{\textnormal{\textbf{s}}}

\newcommand{\bfB}{\textnormal{\textbf{B}}}
\newcommand{\bfU}{\textnormal{\textbf{U}}}

\newcommand{\uf}{\textnormal{uf}}
\newcommand{\fro}{\textnormal{fr}}

\newcommand{\Env}{\textnormal{Env }}
\newcommand{\Spec}{\textnormal{Spec }}

\newcommand{\weight}{\textnormal{wt}}
\newcommand{\ve}{\varepsilon}
\newcommand{\isoto}{\xrightarrow{\sim}}
\newcommand{\LP}{\CL\CP}
\newcommand{\cl}{\textnormal{cl}}

\let\emptyset\varnothing

\title{Reductive monoids and cluster algebras}
\begin{document}

\author[Jinfeng Song]{Jinfeng Song}
\address{Department of Mathematics, The Hong Kong University of Science and Technology, Clear Water Bay, Hong Kong SAR.}
\email{jfsong@ust.hk}

\author[Jeff York Ye]{Jeff York Ye}
\address{Department of Mathematics, National University of Singapore, Singapore.}
\email{e1124873@u.nus.edu}


\begin{abstract}
We show that the coordinate ring of the Vinberg monoid of a simply connected semisimple complex group is an upper cluster algebra. As an application, we construct cluster structures on a large class of flat reductive monoids. After localization, we obtain cluster structures on any connected reductive group whose commutator group is simply connected.
\end{abstract}

\maketitle

\section{Introduction}
\label{sec:intro}
In this paper, all varieties are defined over $\BC$.

\subsection{The Vinberg monoid}
\label{sec:intro:monoid}
An \emph{algebraic monoid} is an affine variety with a monoid structure such that the multiplication map is algebraic. An algebraic monoid $M$ is called \emph{reductive} if its group of units $G(M)$ is reductive. A systematic study of reductive monoids starts from Putcha \cite{Put84} and Renner \cite{Ren85}, and it has attracted many attentions since then. 

In the work of \cite{Vin95}, Vinberg classified reductive monoids using a representation theoretic approach. In particular, associated with any semisimple group $G$, Vinberg defined a universal object $\Env G$, in a family of flat reductive monoids, which is now called the \emph{Vinberg monoid}. The Vinberg monoid $\Env G$ is a reductive monoid whose group of units has the commutator group isomorphic to $G$, and any flat reductive monoid with this property can be obtained from $\Env G$ by a base change. 

Moreover, suppose $r$ is the rank of $G$. Then there is a canonical flat map 
\[
\pi: \Env G\longrightarrow \mathbb{C}^r
\]
onto the affine space, where the generic fibres are isomorphic to the group $G$, and the special fibre $\pi^{-1}(\mathbf{0})$ is the \emph{asymptotic semigroup} considered by Vinberg \cite{Vin95a}.

The Vinberg monoid together with the map $\pi$ have found significant applications in recent developments of the geometric representation theory, see \citelist{\cite{BNS16}\cite{FKM20}} for example.

\subsection{Cluster algebras}
\label{sec:intro:cluster}

The theory of cluster algebras was developed by Fomin and Zelevinsky \cite{FZ02} to study total positivity and dual canonical bases in semisimple groups. The coordinate rings of many varieties related to algebraic groups have shown to be (upper) cluster algebras (See \citelist{\cite{BFZ05}\cite{Wil13}\cite{Sco06}\cite{CGG+25}\cite{GLSS}} for example.). 

In contrast, cluster algebra structures on algebraic monoids have received little attention. It seems that the only nontrivial example is the cluster algebra structure on the space of $n\times n$ matrices from \cite{FWZ20}*{Theorem~6.6.1}.  

The existence of a cluster structure on a variety has a wide range of consequences. For example, it provides a total positivity structure, Poisson structure \cite{GSV10}, and a cluster theoretic canonical basis on its coordinate ring \cite{GHKK18}.

In this paper, we will focus on upper cluster algebras where not all frozen variables are inverted. We borrow the name \emph{partially compactified (upper) cluster algebras} from \cite{GHKK18} to highlight this choice of convention.

\subsection{Main results}
\label{sec:intro:main}

The main result of this paper is the follows.

\begin{thmintro}[Theorem~\ref{thm:mainthm} \& Corollary~\ref{cor:spec}]
\label{thm:1}
Let $G$ be simply connected. The coordinate ring of the Vinberg monoid $\Env G$ is a partially compactified upper cluster algebra where all frozen variables are not invertible. 

The map $\pi$ coincides with the specialization of certain frozen variables. In particular, the asymptotic semigroup is obtained by specializing certain frozen variables to 0.
\end{thmintro}

Combining with the universal property of the Vinberg monoid and the coefficient specialization of cluster algebras, Theorem \ref{thm:1} allows us to construct cluster algebra structures on a large family of reductive monoids. 

\begin{thmintro}
[Theorem~\ref{thm:affineAbelianCluster}]
\label{thm:3}
Let $M$ be a flat reductive monoid. Suppose that the semisimple part of $M$ is simply connected and the abelianization of $M$ is isomorphic to an affine space $\mathbb{C}^k$ (See Section \ref{sec:prelim:Vinberg} for definitions). Then the coordinate ring of $M$ is a partially compactified upper cluster algebra with all frozen variables not invertible.
\end{thmintro}

Let us remark that the assumptions on the reductive monoid $M$ are reasonable. Any set of frozen variables comes with a map $M\to \BC^k$. This should be compatible with the abelianization for a natural cluster structure on $M$. On the other hand, our assumptions on $M$ is flexible enough to cover interesting examples. For example, they include the case where $M$ is the monoid of $n\times n$ matrices. In this case, Theorem \ref{thm:3} recovers the well-known cluster structures, presented in \cite{FWZ20}*{Theorem~6.6.1}.

By inverting the appropriate frozen variables, we obtain cluster algebra structures on any connected reductive group.

\begin{thmintro}[Theorem~\ref{thm:reductiveMonoidCluster}]
\label{thm:2}
Let $H$ be a connected reductive group with simply connected commutator group. Then the coordinate ring of $H$ is a partially compactified upper cluster algebra, whose further partially compactification along all invertible frozen variables is the coordinate ring of a flat reductive monoid with unit group $H$.
\end{thmintro}

In the case when $H$ is semisimple, Theorem \ref{thm:2} was obtained by Qin and Yakimov \cite{QY25}. Independently, Oya \cite{Oya25} proved the same result under the further condition that $H$ is not of type $F_4$. In the semisimple case, there are no additional invertible frozen variables. This is also reflected by the fact that a reductive monoid with semisimple unit group is necessarily a group.

\subsection{Strategy of the proof}
\label{sec:intro:strategy}

The first key ingredient of the proof of Theorem~\ref{thm:1} is the construction of a framed group $\wt{G}$ of $G$. Let $I$ be the set of simple roots of $G$ with $|I|=r$. We define $\wt{I}$ by adding $r$ vertices labeled $i'$ and a simple edge joining $i$ and $i'$ for each $i\in I$. The framed group $\wt{G}$ is defined to be the corresponding Kac--Moody group. We then construct an isomorphism between the unit group of $\Env G$ to the standard Levi subgroup of $\wt{G}$ associated with $I$. Using the cluster structures on double Bruhat cells for Kac--Moody groups constructed by Williams \cite{Wil13} (See also \cite{SW21}), we obtain a cluster structure on an open part $(\Env G)^\circ$ of $\Env G$.

The second essential technique is a comparison between the cluster valuations along frozen variables and the valuation associated with the boundary divisor $\Env G-(\Env G)^\circ$. This is achieved by applying the string parametrization of the dual canonical basis.

We note that the technique of extending root data was inspired by the thickening map introduced by Bao and He \cite{BH24}. Similar ideas appeared in the recent works \citelist{\cite{Qin25}\cite{FH25}} as well.

\subsection{Future problems}
\label{sec:intro:future}

Let us close the introduction by mentioning some future directions.

Firstly, reductive monoids admit standard Poisson structures and quantization \cite{BS25}. We expect that all the results in the current paper can be extended to the quantum setting.

Secondly, we focus only on upper cluster algebras in this paper. For simply-connected simple group $G$ not of type $F_4$, Oya \cite{Oya25} proved that the upper cluster algebra structure on $\mathbb{C}[G]$ coincides with the ordinary cluster algebra in the sense of Fomin and Zelevinsky. A quantum analogue of this coincidence was recently established by Oya, Qin, and Yakimov \cite{OQY25}. It is therefore a natural question whether the upper cluster algebras studied in this paper also coincide with ordinary cluster algebras.

Thirdly, the Vinberg monoid $\Env G$ is isomorphic to the total coordinate ring of the wonderful compactification $\mathbf{X}$ of $G_{ad}=G/Z$. It would be interesting to explore the connections between the cluster algebra structure on $\Env G$ developed in the current paper with the total positivity on $\mathbf{X}$ studied by He \cite{He04}. 

\vspace{.2cm}

\noindent {\bf Acknowledgment: } We are grateful to Huanchen Bao for valuable discussions and suggestions during the writing of this paper. JS is grateful to Fan Qin for helpful discussions at the early stage of the paper. JS is supported by the Glorious Sun Charity Fund.

\section{Preliminaries}
\label{sec:prelim}

\subsection{Cluster algebras}
\label{sec:prelim:cluster}

A \textit{(labeled) seed} $\bfs$ is a quadruple $(J, J_\uf,(\ve_{ij})_{i,j\in J}, (d_i)_{i\in J})$ consists of the following data:
\begin{itemize}
    \item An index set $J$, together with a subset $J_\uf\subset J$. The elements in $J$ are called vertices. The vertices in $J_\uf$ are called \textit{mutable} or \textit{unfrozen}. The vertices in $J_\fro=J-J_\uf$ are called \textit{frozen}.
    \item A skew-symmetrizable matrix $(\ve_{ij})_{i,j\in J}$ over $\BQ$ with $\ve_{ij}\in \BZ$ if at least one of $i,j$ is in $J_\uf$, called the \textit{exchange matrix}.
    \item Integers $d_i$ for $i\in J$ with $\gcd(d_i)=1$ such that $\ve_{ij}d_j=-\ve_{ji}d_i$, called the \textit{symmetrizers}.
\end{itemize}

Let $[a]_+=\max(a,0)$ for $a \in \BQ$. For any $k\in J_\uf$, the \textit{mutation at $k$} gives a new seed $\bfs'=\mu_k(\bfs)$ with the same $J$, $J_\uf$ and $(d_i)_{i \in J}$, and a new exchange matrix $(\ve'_{ij})_{i,j\in J}$ given by
\[
\ve'_{ij} = \begin{cases}-\ve_{ij}, & \text{if } k \in \{i,j\};\\
\ve_{ij}+[\ve_{ik}]_+\ve_{kj}+\ve_{ik}[-\ve_{kj}]_+, &\text{otherwise}.
\end{cases}
\]

We can repeat this process mutating at arbitrary $k\in J_\uf$. A seed $\bfs'$ is said to be \textit{mutation equivalent} to $\bfs$, denoted by $\bfs'\sim \bfs$, if it can be obtained from $\bfs$ by a sequence of mutations.

For any seed $\bfs_0$, the (abstract) upper cluster algebra with initial seed $\bfs_0$ is defined as follows. Fix an ambient field $K = \BC(A_1, \dots, A_n)$, the field of rational functions over $\BC$ in $n$ variables, generated by the formal variables $A_i$ for $i\in J$. The variables $\{A_{i}=A_{i,\bfs_0}\}$ are called the \textit{cluster variables} in the seed $\bfs_0$. The variables $\{A_i \,|\, i\in J_\fro\}$ are  called \textit{frozen variables}.

Given the cluster variables $\{A_{i,\bfs}\}_{i \in J}$ in a seed $\bfs\sim \bfs_0$, the cluster variables $\{A_{i,\bfs'}\}_{i \in J}$ in $\bfs'=\mu_k(\bfs)$ are elements of $K$ defined by 
\[
\begin{cases}
A_{i,\bfs'}=A_{i,\bfs}, &\text{for $i\neq k$};\\
A_{k,\bfs'} = \frac{1}{A_{k,\bfs}}\left(\prod_{i\in J}A_{i,\bfs}^{[\ve_{ki}]_+}+\prod_{i\in J}A_{i,\bfs}^{[-\ve_{ki}]_+}\right), &\text{for $i = k$.}
\end{cases}
\]

Let $\Sigma\subset J_{\text{fr}}$ be a subset of frozen vertices. For any seed $\bfs\sim\bfs_0$, let $\overline{\LP}^\Sigma(\bfs)=\BC[A_{i,\bfs}^{\pm1}, A_j]_{i\in J-\Sigma,j\in \Sigma}$. The \textit{partially compactified upper cluster algebra} $\overline{\CU}^\Sigma(\bfs_0)$ is defined as the intersection
\[
\overline{\CU}^\Sigma(\bfs_0)=\bigcap_{\bfs\sim \bfs_0}\overline{\LP}^\Sigma(\bfs).
\]
We will often omit the initial seed $\bfs_0$ when the context is clear.

For $\Sigma=\emptyset$, we write $\LP(\bfs)=\overline{\LP}^\emptyset(\bfs)$ and $\CU=\overline{\CU}^\emptyset$. For $\Sigma=J_{\text{fr}}$, we write $\overline{\LP}(\bfs)=\overline{\LP}^\Sigma(\bfs)$ and $\overline{\CU}=\overline{\CU}^\Sigma$. 

It is clear that mutation equivalent seeds give canonically isomorphic upper cluster algebras, so the construction only depends on the mutation equivalence class of $\bfs_0$. By the Laurent phenomenon \cite{FZ03}*{Proposition~11.2}, all cluster variables belong to $\overline{\CU}^\Sigma$. 

For $\Sigma'\subset\Sigma\subset J_{\text{fr}}$, the inclusion $\BC[A_j]_{j\in\Sigma'}\hookrightarrow \overline{\CU}^\Sigma$ defines a canonical map
\begin{equation}\label{eq:ff}
\pi^{\Sigma'}:\Spec\overline{\CU}^\Sigma\longrightarrow\BC^{\Sigma'}.
\end{equation} 

For any seed $\bfs\sim \bfs_0$ and $j\in J$, let $\nu^\bfs_j$ be the valuation on $K$ induced from the valuation on $\LP(\bfs)$ given by the vanishing order of $A_{j,\bfs}$. By \cite{Qin25}*{Lemma~2.12}, for $k\in J_\fro$, $\nu^\bfs_j$ is independent of the seed $\bfs$, so we will omit the superscript. In particular, we have
\begin{equation}\label{eq:uos}
\overline{\CU}^\Sigma=\{f\in\CU\mid \nu_i(f)\geq 0,\text{for }i\in \Sigma\}=\CU\cap \overline{\LP}^\Sigma(\bfs).
\end{equation}

\subsection{Minimal Kac--Moody groups}
\label{sec:prelim:KM}

We follow \cite{Kum02} for basics on Kac--Moody groups. Let $r$ be a positive integer, $I=\{1,2,\cdots,r\}$ be a finite index set, and $A=(a_{ij})_{i,j\in I}$ be a symmetrizable generalized Cartan matrix. Take positive integers $d_1,\dots,d_r$, such that $\gcd(d_1,\dots,d_r)=1$ and $d_ia_{ij}=d_ja_{ji}$ for $1\leq i,j,\leq r$. A \textit{Kac--Moody root datum} associated to $A$ is a quintuple
\[
\CD=(I,A,X,Y,(\alpha_i)_{i\in I},(\alpha_i^\vee)_{i\in I}),
\]
where $X$ is a free $\BZ$-module of finite rank with $\BZ$-dual $Y$, and the elements $\alpha_i$ of $X$ and $\alpha_i^\vee$ of $Y$ are such that $\langle \alpha_j^\vee,\alpha_i\rangle=a_{ij}$. Unless otherwise stated, we will always work with simply connected root data, i.e., $Y=\BZ[\alpha_i^\vee]_{i\in I}$ and $X=\BZ[\omega_i]_{i\in I}$, where $\omega_i$ denotes the fundamental weight, that is, $\langle \alpha^\vee_i,\omega_j\rangle=\delta_{ij}$ for $i,j\in I$. We define the set of dominant weights $X^+=\{\lambda\in X\,|\, \langle \alpha_i^\vee,\lambda\rangle \geq0 \text{ for all $i\in I$}\}$, the set of regular dominant weights $X^{++}=\{\lambda\in X\mid \langle\alpha_i^\vee,\lambda\rangle>0\text{ for all $i\in I$}\},$ and the root lattice $R=\BZ[\alpha_i]_{i\in I}\subset X$. Let $R^+=\BZ_{\geq0}[\alpha_i]_{i\in I}\subset R$. We equip $X$ with the standard partial order $\leq$, such that $\lambda\leq \mu$ if and only if $\mu-\lambda\in R^+$.

Let $W=\langle s_i\mid i\in I\rangle$ be the Weyl group associated to $\CD$ with simple reflection $s_i$ for $i\in I$. We denote the length function by $\ell(\cdot)$. The group $W$ acts naturally on both $X$ and $Y$. Let $\Delta^{re}=\{w(\pm\alpha_i)\,|\, i\in I, w\in W\}\subset X$ be the set of real roots. Then $\Delta^{re}=\Delta^{re}_+\sqcup \Delta^{re}_-$ is a disjoint union of the positive and negative real roots.

The \textit{minimal Kac--Moody group} $G$ over $\BC$ associated to the Kac--Moody root datum $\CD$ is the group generated by the torus $T=Y\ox_\BZ \BC^\x$ and the root subgroup $U_\alpha\simeq \BC$ for each real root $\alpha$, subject to the Tits relations \cite{Tit87}. Let $U^+\subset G$ (resp. $U^-\subset G$) be the subgroup generated by $U_\alpha$ for $\alpha\in \Delta^{re}_+$ (resp. $\alpha\in \Delta^{re}_-$). Let $B^\pm\subset G$ be the subgroup generated by $T$ and $U^\pm$, respectively.

For each $i\in I$, we fix isomorphisms $x_i:\BC\to U_{\alpha_i}$, $y_i:\BC\to U_{-\alpha_i}$ such that the maps
\[
\begin{pmatrix}
1 & a \\
0 & 1 \\
\end{pmatrix}\mapsto x_i(a), \quad 
\begin{pmatrix}
b & 0 \\
0 & b^{-1} \\
\end{pmatrix}\mapsto \alpha_i^\vee(b), \quad 
\begin{pmatrix}
1 & 0 \\
c & 1 \\
\end{pmatrix}\mapsto y_i(c),
\]
define a group homomorphism $SL_2(\BC)\to G$. The data $(T,B^+,B^-,x_i,y_i;i\in I)$ is called a \textit{pinning} for $G$.

For each $i\in I$, define $\dot{s}_i=x_i(1)y_i(-1)x_i(1)\in G$. For any $w\in W$ with reduced expression $w=s_{i_1}\cdots s_{i_n}$, we define $\dot{w}=\dot{s}_{i_1}\cdots \dot{s}_{i_n}$. It is known that $\dot{w}$ is independent of the choice of reduced expressions.

For $\lambda\in X^+$, let $V(\lambda)$ be the irreducible $G$-module of highest weight $\lambda$. Following \cite{KP83}, the algebra of \emph{strongly regular functions}, denoted by $\BC[G]$, is defined to be the algebra generated by the matrix coefficients of $V(\lambda)$ as over all $\lambda\in X^+$.

For $i\in I$, fix a highest weight vector $\eta_{\omega_i}\in V(\omega_i)$. The \emph{principal minor} $\Delta^{\omega_i}\in \BC[G]$ is defined by letting $\Delta^{\omega_i}(g)$ to be the coefficient of $\eta_{\omega_i}$ in $g\eta_{\omega_i}$ for any $g\in G$. For $w,w'\in W$, the \emph{generalized minor} $\Delta_{w\omega_i,w'\omega_i}$ is defined to be the strongly regular function on $G$ such that $\Delta_{w\omega_i,w'\omega_i}(g)=\Delta^{\omega_i}(\dot{w}^{-1}g\dot{w'})$ for $g\in G$. Note that $\Delta^{\omega_i}=\Delta_{e\omega_i,e\omega_i}$.

\subsection{Double Bruhat cells}
\label{sec:prelim:Bruhat}

For $u,v\in W$, define 
\[
G^{u,v}=B^+\dot{u}B^+\cap B^-\dot{v}B^-
\]
to be the \emph{double Bruhat cell}. It is known that $G^{u,v}$ is a complex variety of dimension $l(u)+l(v)+r$. We recall the cluster structure on the coordinate algebra $\BC[G^{u,v}]$ following \cites{Wil13}.

A \emph{double reduced word} $\bfi$ for $(u,v)$ is a reduced word $(i_1,\dots,i_l)$ for $(u,v)$ in $W\x W$, in the alphabet $[-r,-1]\sqcup[1,r]$, where the simple reflections of the first copy of $W$ are denoted by $[-r,-1]$ and those for the second copy by $[1,r]$. A double reduced word $\bfi$ is called \emph{unshuffled} if $(i_1,\dots,i_{l(v)})$ is a reduced word for $v$ and $(i_{l(v)+1},\dots,i_l)$ is a reduced word for $u$.

Let $\bfi$ be double reduced word for $(u,v)$ and set $l=l(u)+l(v)$. We define a seed $\bfs(\bfi)$ as follows. The index set is $J=[-r,-1]\sqcup[1,l]$. Set $i_{-k}=-k$ for $k\in [1,r]$. We say that a vertex $k\in J$ is on \emph{level} $|\bfi_k|$. For $k\in J$, set
\[
k^+=\min\{m\in J\mid m>k, |i_m|=|i_k|\},
\]
and set $k^+=l+1$ if there is no such $m$. An index $k\in J$ is frozen if either $k<0$ or $k^+>l$. The exchange matrix $(\ve_{ij})_{i,j\in J}$ is defined by 
\[
    \begin{split}
        \ve_{jk}= \frac{a_{|i_j|,|i_k|}}{2} \left(\right.&d_j[j=k^+]-d_k[j^+=k]\\
        &+d_j[k<j<k^+][j>0]-d_{j^+}[k<j^+<k^+][j^+\leq m]\\
        &\left.-d_k[j<k<j^+][k>0]+d_{k^+}[j<k^+<j^+][k^+\leq m]\right),
    \end{split}
\]
where $d_k=d_{|i_k|}$ and $[P(j,k,\dots)]$ is either $1$ or $0$ depending on whether $P(j,k,\dots)$ is true or false. The elements $d_k$ for $k\in J$ also serve as the symmetrizers for the seed $\bfs(\bfi)$.

For the fixed reduced word $\bfi$ for $(u,v)$, and $k\in [1,l]$, we denote
\[
u_{\leq k}=\prod_{\substack{t=1,\dots,k\\i_t<0}}s_{|i_t|},\qquad v_{>k}=\prod_{\substack{t=l,\dots,k+1\\i_t>0}}s_{|i_t|}.
\]
For $k\in[-r,-1]$, we set $u_{\leq k}=e$ and $v_{>k}=v^{-1}$. Now for $k\in J$ set
\begin{equation}\label{eq:dki}
\Delta(k;\bfi)=\Delta_{u_{\leq k}\omega_{|i_k|},v_{>k}\omega_{|i_k|}}.
\end{equation}

The generalized minors are viewed as functions on $G^{u,v}$ by restriction.

\begin{theorem}[\cite{Wil13}*{Theorem~4.9(1)}]
\label{thm:DBCcluster}
For any $(u,v)\in W\x W$ and double reduced word $\bfi$ for $(u,v)$, there is an isomorphism
\[
a_\bfi:\mathcal{U}(\bfs(\bfi))\longrightarrow \BC[G^{u,v}]
\]
given by $A_{k,\bfs(\bfi)}\mapsto \Delta(k;\bfi)$.
\end{theorem}

By \cite{SW21}*{Proposition 3.25}, different choices of $\bfi$ give mutation equivalent seeds $\bfs(\bfi)$, so the cluster structure on $G^{u,v}$ is independent of the choice of $\bfi$.

\subsection{The Vinberg monoid}
\label{sec:prelim:Vinberg}

In this section, we assume that $G$ is semisimple and simply connected. For $\lambda\in X^+$, let $V(\lambda)$ be the irreducible representation of $G$ with highest weight $\lambda$, and let $V(\lambda)^*$ be the dual $G$-module. By the Peter--Weyl theorem, there is a canonical isomorphism
\begin{equation}
\label{eq:PW}
\BC[G]\cong \bigoplus_{\lambda\in X^+}V(\lambda)^*\ox V(\lambda)
\end{equation}
as $G\x G$-modules. For $\mu\in X$, let $\BC[G]_\mu=V(\mu)^*\otimes V(\mu)$ and $\BC[G]_{\leq\mu}\subset \BC[G]$ be the subspace spanned by $V^*(\lambda)\ox V(\lambda)$ for $\lambda\in X^+$ and $\lambda\leq\mu$. Note that $\BC[G]_{\leq\mu'}\BC[G]_{\leq\mu''}\subseteq \BC[G]_{\leq \mu'+\mu''}$ for $\mu',\mu''\in X$.

Following \cite{Vin95}, let $R(G)$ be the subalgebra of $\BC[G]\ox \BC[T]$ given by
\[
R(G)=\bigoplus_{\mu\in X}\BC[G]_{\leq \mu}e^\mu,
\]
and define the \emph{Vinberg monoid} $\Env G$ to be the affine variety $\Spec R(G)$.

We next recall a universal property of the Vinberg monoid. For any reductive monoid $M$, let $G(M)$ denote the group of units of $M$. The \emph{semisimple part} $G_M$ of $M$ is defined to be the commutator group of $G(M)$. The \emph{abelianization} $A(M)$ of $M$ is the GIT quotient of $M$ by the two-sided action of $G_M$. 

There is a canonical surjection
\[
\pi: M\to A(M).
\]
A normal reductive monoid $M$ is said to be \emph{flat} if $\pi$ is flat with reduced and irreducible fibres. A homomorphism of reductive monoids induces a homomorphism of their abelianizations.

For the Vinberg monoid, it is clear that $G(\Env G)= G\times^ Z T$, $G_{\Env G}= G$, and the abelianization $A(\Env G)$ is isomorphic to the affine space $\BC^I$. The canonical surjection 
\begin{equation}
\label{eq:vd}
\pi: \Env G\to \BC^I
\end{equation}
is induced by the inclusion $\BC[e^{\alpha_i}]_{i\in I}\hookrightarrow \BC[\Env G]$. The map \eqref{eq:vd} is called the \emph{Vinberg degeneration} of the semi-simple group $G$. It is known that $\pi$ is flat and $\pi^{-1}((1,\cdots,1))\simeq G$ \cite{Vin95}*{Theorem~3 \& Theorem~4}.

Vinberg showed that $\Env G$ satisfies the following universal property. For any flat reductive monoid $M$ and any isomorphism $\varphi_0:G_M\to G$, there is a unique homomorphism $\varphi:M\to \Env G$ extending $\varphi_0$, and we have an isomorphism
\[
M\simeq \Env G\x_{A(\Env G)} A(M).
\]

\begin{example}
\label{eg:part1}
Our running example will be $G=SL_2$. We see that the Vinberg monoid for $SL_2$ is $M_2$, the space of $2$ by $2$ matrices, and the group of units is $G\x^ZT\simeq GL_2$.
\end{example}

\subsection{The string parametrization}
\label{sec:prelim:QG}

In this section, we recall the string parametrization of canonical basis, which will be used in the proof of Theorem \ref{thm:1}. We firstly recall the basics on quantum groups and canonical bases, following \cite{Lus93}.

Let $\bfU=\mathbb{Q}(q)\langle E_i,F_i,K_\mu\mid i\in I,\mu\in Y\rangle$ be the Drinfeld--Jimbo quantum group associated with the root datum of $G$ (cf. \cite{Lus93}*{Section~3.1.1}. Then $\bfU$ admits a triangular decomposition
\[
\bfU\simeq \bfU^+\ox \bfU^0\ox \bfU^-,
\]
where $\bfU^+,\bfU^0,\bfU^-$ are generated by $E_i$ $(i\in I)$, $K_\mu$ $(\mu\in Y)$, and $F_i$ $(i\in I)$ respectively. For any $\lambda\in X^+$, we denote by $V_q(\lambda)$ the simple $\bfU$-module with highest weight $\lambda$. For any $\lambda\in X^+$, fix a highest weight vector $\eta_\lambda$ in $V_q(\lambda)$. 

Let $\bfB$ be canonical basis for $\bfU^-$ \cite{Lus93}. For any $\lambda\in X^+$, there is a basis $B_\lambda$ of $V_q(\lambda)$ and a subset $\bfB(\lambda)\subset \bfB$, such that the map $b\mapsto b\eta_\lambda$ sends $\bfB-\bfB(\lambda)$ to 0 and restricts to a bijection $\bfB(\lambda)\to B_\lambda$. For $\lambda\in X^+$and $\mu\in W\lambda$, we write $\eta_\mu\in V_q(\lambda)$ to be the unique element in $B_\lambda$ of weight $\mu$. For $b\in B_\lambda$, we write $b^*\in V_q(\lambda)^*$ to be the linear form such that $b^*(b')=\delta_{b,b'}$ for $b'\in B_\lambda$.

For $i\in I$ and $\lambda\in X^+$, let $\wt{e}_i$, $\wt{f}_i$ be the Kashiwara operators on $V_q(\lambda)$ (cf. \cite{Kas91}*{Section 2.2}). We also have the Kashiwara operators $\wt{e}_i,\wt{f}_i:\bfB\rightarrow \bfB\sqcup\{0\}$, such that for any $b\in \bfB(\lambda)$, one has $\wt{e}_i(b)\eta_\lambda-\wt{e}_i(b\eta_\lambda)$, $\wt{f}_i(b)\eta_\lambda-\wt{f}_i(b\eta_\lambda)\in q^{-1}\BZ[q^{-1}]B_\lambda$. 

Let $(-,-)$ be the bilinear form defined in \cite{Lus93}*{Chapter 1}, which we will identify as a bilinear form $\bfU^-\x \bfU^-\to \BQ(q)$. Let $\bfB^*=\{b^*\,|\, b\in \bfB\}\subset \bfU^-$ be the dual canonical basis such that $(b^*,b')=\delta_{b,b'}$ for $b,b'\in\mathbf{B}$. Let $d_{b,b'}^{b''}$ be the structure coefficients such that
\[
b^*{b'}^*=\sum_{b''\in \bfB} d_{b,b'}^{b''} {b''}^*.
\]
It is well-known that $d_{b,b'}^{b''}\in\BZ[q,q^{-1}]$.

We next recall the string parametrization for $\bfB$, following \cite{Cal02} (see also \cites{BZ01,Lit98}). Let $\bfi=(i_1,\dots,i_N)$ be a reduced expression for $w_0$. For any $b\in \bfB$ and $i\in I$, set 
\[
\ve_i(b)=\max\{r \mid {\wt{e}}^r_i(b)\neq 0\},\quad\Lambda_i(b)=\wt{e}_i^{\ve_i(b)}(b),
\]
and 
\[
c_\bfi(b)=(\ve_{i_1}(b),\ve_{i_2}(\Lambda_{i_1}(b)),\dots, \ve_{i_N}(\Lambda_{i_{N-1}}\cdots\Lambda_{i_1}(b)))\in\BN^N.
\]

Note that we reversed the reduced expression for $w_0$ in \cite{Cal02} to make some formulations simpler later on.

It is known that $c_\bfi:\bfB\rightarrow \BN^N$ is injective (cf. \cite{Lit98}*{Section 1}). Moreover, we have the following theorem for the triangularity of the multiplication of $\bfB^*$.

\begin{theorem}[\cite{Cal02}*{Theorem 2.3}]
\label{thm:productCB1}
Let $b,b',b''\in \bfB$. Then $d_{b,b'}^{b''}\neq 0$ implies
\[
c_\bfi(b'')\leq c_\bfi(b)+c_\bfi(b')
\]
in the lexicographical order of $\BN^{\ell(w_0)}$. If the equality holds, then $d_{b,b'}^{b''}$ is a power of $q$.
\end{theorem}

For $\lambda\in X^+$, let $V_{q}(\lambda)_\mathbb{Z}\subset V_q(\lambda)$ be the $\mathbb{Z}[q,q^{-1}]$-submodule spanned by $B_\lambda$. Then $\mathbb{C}\otimes _{\mathbb{Z}[q,q^{-1}]}V_{q}(\lambda)_\mathbb{Z}\cong V(\lambda)$, where $\mathbb{C}$ is viewed as a $\mathbb{Z}[q,q^{-1}]$-module via $q\mapsto 1$. By abuse of notation we denote by $B(\lambda)\subset V(\lambda)$ the canonical basis of $V(\lambda)$. In the rest of the paper we shall only consider the classical case $q=1$.

\section{Cluster structure on the Vinberg monoid}
\label{sec:cluster}
In this section, we assume that $G$ is semisimple and simply connected group. Let $\CD=(I,A,X,Y,(\alpha_i)_{i\in I},(\alpha_i^\vee)_{i\in I})$ be the associated root datum. 

\subsection{Parametrization of basis}
\label{sec:cluster:weight}

Recall the Peter-Weyl decomposition
\[
\BC[G]=\bigoplus_{\lambda\in X^+} \BC[G]_{\lambda},
\]
where $\BC[G]_{\lambda}=V(\lambda)^*\ox V(\lambda)$ for $\lambda\in X^+$. Then any $f\in \BC[G]$ can be uniquely written as
\[
f=\sum_{\lambda\in X^+}f_\lambda
\]
where $f_\lambda\in \BC[G]_\lambda$. We write $\weight(g)=\lambda$ for $g\in V(\lambda)^*\ox V(\lambda)$. For $\lambda\in X^+$, let $\pi_\lambda:\BC[G]\rightarrow \BC[G]_\lambda$ be the canonical projection. 

For $\mu\in X^+$, there is a unique $G$-equivariant isomorphism
\begin{equation}
    \beta_\mu:V(\mu)\overset{\sim}{\longrightarrow} V(-w_0\mu)^*,\qquad \eta_\mu\mapsto \eta_{-\mu}^*.
\end{equation}
For $\mu\in X^+$, let $B_\mu^*=\{b^*\mid b\in B_\mu\}$ and $B_\mu'=\beta_\mu^{-1}(B_{-w_0\mu}^*)$. Then 
\begin{equation}
    \wt{B}_\mu=\{b\ox b'\mid b\in B_\mu^*,b'\in B'_\mu\}
\end{equation}
is a basis of $\BC[G]_\mu$. Let $\wt{B}=\sqcup_{\mu\in X^+}\wt{B}_\mu$. Then $\wt{B}$ is a basis of $\BC[G]$.

\begin{lem}
\label{lem:MinorIsCB}
For any $u,v\in W$ and $i\in I$, the generalized minor $\Delta_{u\omega_i,v\omega_i}$ belongs to $\wt{B}$ up to a $\pm$ sign.
\end{lem}

\begin{proof}
We have
\begin{equation}
\label{eq:MinorTensor}
\Delta_{u\omega_i,v\omega_i}=\eta_{u\omega_i}^*\ox \eta_{v\omega_i}=\pm\eta_{u\omega_i}^*\ox \beta^{-1}(\eta_{vw_0(-w_0\omega_i)}^*).
\end{equation}
The result then follows from $\eta_{u\omega_i}\in B_{\omega_i}$ and $\eta_{vw_0(-w_0\omega_i)}\in B_{-w_0\omega_i}$.
\end{proof}

Let $\bfj,\bfj'$ be two reduced expressions of $w_0$. For $\mu\in X^+$, $b\in \bfB(\mu)$, and $b'\in \bfB({-w_0\mu})$, set
\begin{equation}
    c_{\bfj,\bfj'}(\pi_\mu(b)^*\ox\beta_\mu^{-1}(\pi_{-w_0\mu}(b')^*))=(c_{\bfj}(b),c_{\bfj'}(b'), \mu)\in\BN^N\x\BN^N\x X.
\end{equation}
By the string parametrization of $\bfB$ (Section \ref{sec:prelim:QG}), the map
\[
c_{\bfj,\bfj'}:\wt{B}\rightarrow \BN^N\x\BN^N\x X
\]
is injective, which we will also call the string parametrization.

We now use the string parametrization to study the cluster variables for $\BC[G]$. Fix an unshuffled double reduced expression $\bfi=(\bfj,-\bfj')$ for $(w_0,w_0)$. Recall the generalized minors 
\[
\Delta(k;\bfi)=\Delta_{u_{\leq k}\omega_{|i_k|},v_{>k}\omega_{|i_k|}} \in \BC[G]
\]
from \eqref{eq:dki}, which we will simply denote by $\Delta(k)$.

For any $i\in I$, let $\overline{i}\in I$ be such that $s_{\overline{i}}=w_0s_iw_0$. For $\bfi=(i_1,\cdots,i_k)$, let $\overline{\bfi}=(\overline{i_1},\cdots,\overline{i_k})$. By Lemma~\ref{lem:MinorIsCB}, we can talk about the string parametrization of the $\Delta(k)$ by omitting the sign. The computation is done in the following lemma.

\begin{lemma}
\label{lem:MinorString}
We have
\begin{enumerate}
    \item For $k\in [-r,-1]$, $c_{\bfj,\overline{\bfj'}}(\Delta(k))=(\mathbf{0},\mathbf{0},\omega_{|k|})$.
    \item For $k\in [1,N]$, $c_{\bfj,\overline{\bfj'}}(\Delta(k))=(\mathbf{0},\mathbf{x}^k,\omega_{j_k'})$, where $\mathbf{x}^k=(x_1^k,\dots,x_N^k)$ satisfies $x_l^k=0$ for $l>k$ and $x_k^k=1$.
    \item For $k\in [N+1,2N]$, $c_{\bfj,\overline{\bfj'}}(\Delta(k))=(\mathbf{y}^k,\mathbf{z}^k,\omega_{j_k})$, where $\mathbf{y}^k=(y_1^k,\dots,y_N^k)$ satisfies $y_l^k=0$ for $l>k$ and $y_k^k=1$.
\end{enumerate}
\end{lemma}

\begin{proof}
The lemma follows from \eqref{eq:MinorTensor} and the definitions of $u_{\leq k}$, $v_{>k}$, and $c_{\bfj,\overline{\bfj'}}$.
\end{proof}

For $\mu_1,\mu_2\in X^+$, let 
\[
m_{\mu_1,\mu_2}:V(\mu_1)^*\ox V(\mu_2)^*\longrightarrow V(\mu_1+\mu_2)^*
\]
be the unique $G$-quivariant map such that $\eta_{\mu_1}^*\ox \eta_{\mu_2}^*\mapsto \eta_{\mu_1+\mu_2}^*$, and let
\[
r_{\mu_1,\mu_2}:V(\mu_1)\ox V(\mu_2)\longrightarrow V(\mu_1+\mu_2)
\] 
be the unique $G$-equivariant map such that $\eta_{\mu_1}\ox \eta_{\mu_2}\mapsto \eta_{\mu_1+\mu_2}$.

\begin{lemma}
\label{lem:TensorMultiplication}
Let $\mu_i\in X^+$, $\xi_{i}\in V(\mu_i)^*$, $v_{i}\in V(\mu_i)$, and $f_i=\xi_{i}\ox v_{i}\in \BC[G]_{\mu_i}$, for $i=1,2$. Then
\[ 
\pi_{\mu_1+\mu_2}(f_1 f_2)=m_{\mu_1,\mu_2}(\xi_1\ox \xi_2)\ox r_{\mu_1,\mu_2}(v_1\ox v_2).
\]
\end{lemma}

\begin{proof}
For any $g\in G$,
\[
(f_1f_2)(g)=\xi_1(gv_1)\xi_2(gv_2)
=(\xi_1\ox \xi_2)(g(v_1\ox v_2)).
\]

Note that $m_{\mu_1,\mu_2}(\xi_1\ox \xi_2)$ is the restriction of $\xi_1\ox \xi_2$ onto the unique $V(\mu_1+\mu_2)$-component of $V(\mu_1)\ox V(\mu_2)$. Hence 
\begin{align*}
\pi_{\mu_1+\mu_2}(f_1 f_2)(g)&=(\xi_1\ox \xi_2)(g(r_{\mu_1,\mu_2}(v_1\ox v_2))\\
&=m_{\mu_1,\mu_2}(\xi_1\ox \xi_2)\ox r_{\mu_1,\mu_2}(v_1\ox v_2)(g).
\end{align*}
This completes the proof.
\end{proof}

Define the partial order $\leq$ on $\BN^N\x\BN^N\x X$ by setting $(\mathbf{x}_1,\mathbf{x}_2,\mu_1)\leq (\mathbf{y}_1,\mathbf{y}_2,\mu_2)$ if and only if either $\mu_1<\mu_2$, or $\mu_1=\mu_2$ and $\mathbf{x}_1\leq\mathbf{y}_1$, $\mathbf{x}_2\leq\mathbf{y}_2$ in the lexicographical order.

For $b,b'\in\wt{B}$, write
\[
b\cdot b'=\sum_{b''\in\wt{B}}\wt{d}^{b''}_{b,b'}b''\qquad \text{ in }\BC[G],
\]
where $\wt{d}_{b,b'}^{b''}\in\BC$.

\begin{lemma}
\label{lem:MinorPO}
For $b,b',b''\in\wt{B}$ and two reduced expressions $\bfj, \bfj'$ of $w_0$, if $\wt{d}_{b,b'}^{b''}\neq 0$, then $c_{\bfj,\bfj'}(b'')\leq c_{\bfj,\bfj''}(b)+c_{\bfj,\bfj'}(b')$. When equality holds, $\wt{d}_{b,b'}^{b''}=1$.
\end{lemma}

\begin{proof}
    The lemma follows from Theorem \ref{thm:productCB1} and Lemma \ref{lem:TensorMultiplication}. 
\end{proof}

For $\bfd=(d_k)_{k\in J}\in \BN^J$, let $\Delta(\bfd)=\prod_{k\in J}\Delta(k)^{d_k}$, $\weight(\Delta(\bfd))=\sum_{k\in J}d_j\weight(\Delta(k))$, and $\overline{\Delta(\bfd)}=\pi_{\weight(\Delta(\bfd))}(\Delta(\bfd))$.

\begin{lemma}\label{lem:MonomialLinIndep}
    Elements $\overline{\Delta(\bfd)}$ $(\bfd\in\BN^J)$ are linearly independent. 
\end{lemma}

\begin{proof}
    For $\bfd\in\BN^J$, let us write $\overline{\Delta(\bfd)}=\sum_{b\in\wt{B}}Q_{\bfd,b}b$ in $\BC[G]$ where $Q_{\bfd,b}\in\BC$. By Lemma \ref{lem:MinorPO} and Lemma \ref{lem:MinorString}, $Q_{\bfd,b}=0$ unless $c_{\bfj,\overline{\bfj'}}(b)\leq \sum_{k\in J}d_kc_{\bfj,\overline{\bfj'}}(\Delta(k))$, and when the equality holds, $Q_{\bfd,b}=\pm1$. By Lemma~\ref{lem:MinorString}, when $\bfd\neq \bfd'$, the leading terms $\sum_{k\in J}d_kc_{\bfj,\overline{\bfj'}}(\Delta(k))\neq \sum_{k\in J}d'_kc_{\bfj,\overline{\bfj'}}(\Delta(k))$, so we get the linear independence.
\end{proof}

Meanwhile, from the decomposition
\[
\BC[\Env G]\simeq \bigoplus_{\mu\in X} \left(\bigoplus_{\lambda\leq \mu}V(\lambda)^*\ox V(\lambda)\right)e^{\mu},
\]
any $f\in \BC[\Env G]$ can be uniquely written as
\[
f=\sum_{\lambda\in X^+,\gamma\in R^+}f_{\lambda,\gamma}e^{\lambda+\gamma}
\]
where $f_{\lambda,\gamma}\in V(\lambda)^*\ox V(\lambda)$. We also let
\[
f_\lambda=\sum_{\gamma\in R^+}f_{\lambda,\gamma}e^{\lambda+\gamma}.
\]
The similar decomposition applies to $\BC[G\x^Z T]$, with $R^+$ replaced with $R$.

For any $i\in I$, let $\nu_i$ be the valuation on $\BC[\Env G]$ given by
\begin{equation}
\label{eq:VinbergValuation}
\nu_i(f)=\min\{\langle \gamma,\omega_i^\vee\rangle\mid f_{\lambda,\gamma}\neq0\},    
\end{equation}
where $\{\omega_i^\vee\}_{i\in I}$ are the fundamental coweights of $G$. The valuation $\nu_i$ extends to the field of fractions $\mathbb{C}(\Env G)$. We have
\[
\BC[\Env G]=\{f\in \BC[G\x^ZT]\mid \nu_i(f)\geq0 \text{ for all $i\in I$}\}.
\]

\subsection{Framed groups}
\label{sec:cluster:framed}

Recall that $G$ is a semisimple simply connected group. Let $\CD$ be the root datum associated with $G$. 

\begin{definition}
    The \emph{framed root datum} $\wt{\CD}=(\wt{I},\wt{A},\wt{X},\wt{Y},(\wt{\alpha}_i)_{i\in \wt{I}},(\wt{\alpha}_i^\vee)_{i\in\wt{I}})$ is the simply connected Kac--Moody root datum such that:
    \begin{itemize}
        \item The index set $\wt{I}=I\sqcup I'$, where $I'$ is a copy of $I$ with indices denoted by $i'$ for $i\in I$;
        \item The generalized Cartan matrix $\wt{A}$ is given by $\wt{a}_{ij}=a_{ij}$ for $i,j\in I$, $\wt{a}_{i'j}=\wt{a}_{ij'}=-\delta_{ij}$, and $\wt{a}_{i'j'}=2\delta_{ij}$, for $i,j\in I$.
    \end{itemize}
    The \textit{framed group} $\wt{G}$ of $G$ is defined to be the minimal Kac--Moody group associated with $(\wt{I},\wt{A})$.
\end{definition}

Let $(\wt{\omega}_i)_{i\in \wt{I}}$ be the set of fundamental weights. Let $\wt{W}$ be the Weyl group associated with $\wt{\CD}$. Let $\wt{W}_I$ be the parabolic subgroup of $\wt{W}$ and $\wt{L}_I$ be the standard Levi subgroup of $\wt{G}$ corresponding to $I$. We have the natural identification $\iota:W\to \wt{W}_I$ and the embedding $\iota:G\to \wt{L}_I$.

Let $Z\subset G$ be the center of $G$, and let
$G\x ^Z T=(G\x T)/\widetilde{Z}$, where $\widetilde{Z}=\{(z,z^{-1})\mid z\in Z\}\subset G\x T$.

Define $\varphi:G \x^Z T\to \wt{G}$ to be the map given by
\[
[(g,t)]\mapsto \iota(gt)\prod_{i\in I}\wt{\alpha}_{i'}^\vee(\alpha_i(t)).
\]

\begin{lemma}
\label{lem:BruhatIso}
The map $\varphi$ is an isomorphism onto $\wt{L}_I$. Furthermore, for any $u,v\in W$, $\varphi$ restricts to an isomorphism
\[
G^{u,v} \x^Z T\isoto \wt{G}^{\iota(u),\iota(v)}.
\]
\end{lemma}

\begin{proof}

We firstly show that $\varphi$ is injective. Suppose  $\varphi(g,t)=\varphi(g',t')$. Then we have $\alpha_i(t)=\alpha_i(t')$ for all $i\in I$, so $t$ and $t'$ differ by an element in $Z$. We also get $gt=g't'$, so $(g,t)=(g',t')$ in $G \x^Z T$.

We next show that the image is exactly $\wt{L}_I$. For any $\wt{g}$ in $\wt{L}_I$, we can write $\wt{g}=\iota(g)\wt{t}$ for some $g\in G$ and $\wt{t}\in \wt{T}$, where $\wt{T}\subset\wt{G}$ is the maximal torus. We can further choose $t\in T$ such that
\[
\wt{t}=\iota(t')\prod_{i\in I}\wt{\alpha}_{i'}^\vee(\alpha_i(t))
\]
for some $t'\in T$. Then $\wt{g}=\varphi(gt't^{-1},t)$, so $\varphi$ is surjective. By Zariski's main theorem, we conclude that $\varphi$ is an isomorphism.

The statement on the double Bruhat cells can be verified by directly checking $\iota(B^+\dot{u}B^+)\subset \wt{B}^+\iota(\dot{u})\wt{B}^+$ and $\iota(B^-\dot{v}B^-)\subset \wt{B}^-\iota(\dot{v})\wt{B}^-$ for $u,v\in W$.
\end{proof}

Let $\bfi$ be a double reduced word for $(w_0,w_0)$. Then via the inclusion $\iota$, $\bfi$ is also a double reduced word for $(\iota(w_0),\iota(w_0))\in \wt{W}\x\wt{W}$. We denote by $\wt{\bfs}(\bfi)$ the seed for the cluster structure on $\wt{G}^{\iota(w_0),\iota(w_0)}$ associated with the double reduced word $\bfi$ (See Section \ref{sec:prelim:Bruhat}). Note that the cluster variables are of the form $\Delta_{\iota(u)\wt{\omega}_i,\iota(v)\wt{\omega}_i}$ and $\Delta_{\wt{\omega}_{i'},\wt{\omega}_{i'}}$ for some $u,v\in W$ and $i\in I$.

By Lemma~\ref{lem:BruhatIso}, the map $\varphi$ restricts to an isomorphism 
\[
G^{w_0,w_0}\x ^Z T\overset{\sim}{\rightarrow} \wt{G}^{\iota(w_0),\iota(w_0)},
\]
which we still denote by $\varphi$. Note that $X$ is isomorphic to the group of characters on $T$. For any $\mu\in X$, write $e^\mu$ to denote the associated character on $T$, which is a regular function on $T$.

\begin{lemma}
\label{lem:PullbackCluster}
Let $\varphi^*$ denote the pullback by $\varphi$. For $u,v\in W$ and $i\in I$, we have
\begin{align*}
\varphi^*(\Delta_{\iota(u)\wt{\omega}_i,\iota(v)\wt{\omega}_i})&=\Delta_{u\omega_i,v\omega_i}e^{\omega_i}, & \varphi^*(\Delta_{\wt{\omega}_{i'},\wt{\omega}_{i'}})&=e^{\alpha_i}.
\end{align*}
\end{lemma}

\begin{proof}
For $t\in T$, we write $\wt{t}=\iota(t)\prod_{i\in I}\wt{\alpha}_{i'}^\vee(\alpha_i(t))\in \wt{G}$. Since $\wt{a}_{ij'}=-\delta_{ij}$ for $i,j\in I$, we have
\[
\wt{\alpha}_i(\wt{t})=\alpha_i(t)\alpha_i(t)^{-1}=1,
\]
for all $i\in I$. Therefore $t$ commutes with $\iota(G)$.

For any $g\in G$ and $t\in T$, we have
\begin{align*}
\Delta_{\iota(u)\wt{\omega}_i,\iota(v)\wt{\omega}_i}(\varphi(g,t))&=\Delta^{\wt{\omega}_i}(\dot{u}^{-1}\iota(g)\wt{t}\dot{v})\\
&=\Delta^{\wt{\omega}_i}(\dot{u}^{-1}\iota(g)\dot{v}\wt{t})\\&=\Delta^{\omega_i}(\dot{u}^{-1}g\dot{v})\wt{\omega}_i(\wt{t})\\
&=\Delta_{u\omega_i,v\omega_i}(g)\omega_i(t),
\end{align*}
and
\[
\Delta_{\wt{\omega}_{i'},\wt{\omega}_{i'}}(\varphi(g,t))=\Delta^{\wt{\omega}_{i'}}(\wt{t})=\alpha_i(t).
\]
Therefore we conclude the proof.
\end{proof}

\begin{example}
Continuing with the running example $G=SL_2$, we have $\wt{G}=SL_3$. The Levi subgroup $\wt{L}_I$ contains matrices of the form $\begin{pmatrix}
* & * & 0\\
* & * & 0\\
0 & 0 & *\\
\end{pmatrix}$, and $\wt{L}_I\simeq GL_2$. The isomorphism $G\x^ZT\simeq \wt{L}_I$ is explicitly given by
\[
\left(\begin{pmatrix}
x_{11} & x_{12}\\
x_{21} & x_{22}
\end{pmatrix},\begin{pmatrix}
t & 0\\
0 & t^{-1}
\end{pmatrix}\right)\mapsto \begin{pmatrix}
x_{11}t & x_{12}t & 0\\
x_{21}t & x_{22}t & 0\\
0 & 0 & t^{-2}
\end{pmatrix}.
\]
\end{example}

\subsection{Partially compactifications}
\label{sec:cluster:partial}

In this subsection, we perform partially compactifications on the seed $\wt{\bfs}(\bfi)$ and prove the first main theorem of this paper.

For a double reduced word $\bfi$ for $(w_0,w_0)$ which is also viewed as a double reduced word for $(\iota(w_0),\iota(w_0))$, let $\wt{\bfs}(\bfi)$ be the seed for the cluster structure on $\wt{G}^{\iota(w_0),\iota(w_0)}$ as before. Let $\wt{J}$ be the index set for $\wt{\bfs}(\bfi)$. The frozen vertices $\wt{J}_\fro$ can be identified with $I^-\sqcup I'\sqcup I^+$, where $I^-$ are the negative vertices on levels in $I$, $I'$ are the negative vertices on levels in $I'$, and $I^+$ are the positive frozen vertices on levels in $I$. Recall the isomorphism $a_\bfi:\CU(\wt{\bfs}(\bfi))\overset{\sim}{\rightarrow}\BC[\wt{G}^{\iota(w_0),\iota(w_0)}]$ from Theorem~\ref{thm:DBCcluster}.

Let $\Psi_\bfi=\varphi^*\circ a_\bfi$ denote the composition
\[
\Psi_\bfi:\CU(\wt{\bfs}(\bfi))\overset{\sim}{\longrightarrow}\BC[G^{w_0,w_0}\x ^Z T].
\]

We have two valuations on $\BC[G^{w_0,w_0}\x ^Z T]$, one coming from the cluster structure via $\Psi_\bfi$, one coming from the geometry of the Vinberg monoid. For any $i\in I$, let $\nu^\cl_i$ be the cluster valuation on $\CU(\wt{\bfs}(\bfi))$ for the unique frozen variable on level $i'$, as defined in Section \ref{sec:prelim:cluster}. Since the seeds for different choices of $\bfi$ are mutation equivalent, $\nu^\cl_i$ is independent of the choice of $\bfi$. We also have the valuation $\nu_i$ by restricting from \eqref{eq:VinbergValuation}.

\begin{prop}
\label{prop:ValuationAgreeVinberg}
For any $i\in I$, the valuations $\nu^\cl_i$ and $\nu_i$ agree.
\end{prop}

\begin{proof}
We will work with an unshuffled double reduced word for $(w_0,w_0)$. By the explicit description of the cluster variables in $\wt{\bfs}(\bfi)$, $\nu^\cl_i$ and $\nu_i$ agree on cluster variables.

To show that $\nu^\cl_i=\nu_i$, it suffices to check on any cluster polynomial $f$ in $\wt{\bfs}(\bfi)$. Note that $f$ can be written as
\[
f=\sum_{\lambda} P_{\lambda,\gamma} e^{\lambda+\gamma},
\]
where $P_{\lambda,\gamma}$ are cluster polynomials in $\bfs(\bfi)$ with $\weight(P_{\lambda,\gamma})=\lambda$. We have
\begin{align*}
\nu_i(f)&\geq \min\{\nu_i(P_{\lambda,\gamma}e^{\lambda+\gamma})\mid P_{\lambda,\gamma}\neq0\}\\
&\geq\min\{\langle \gamma,\omega_i^\vee\rangle\mid P_{\lambda,\gamma}\neq0\}\\
&=\nu^\cl_i(f).
\end{align*}
To show the converse, it suffices to show that $\nu_i(f)=0$ whenever $\langle \gamma,\omega^\vee_i\rangle=0$ for any $P_{\lambda,\gamma}\neq 0$. Recall that we denote by $f_\lambda$ the projection of $f$ onto the Peter-Weyl component $\BC[G]_\lambda$.   Let $\lambda_0$ be a maximal weight among the lambdas such that $f_\lambda\neq0$. By Lemma \ref{lem:MonomialLinIndep}, the projection $\overline{P_{\lambda_0,\gamma}}$ to $\BC[G]_{\lambda_0}$ is nonzero. The maximality of $\lambda_0$ ensures that $f_{\lambda_0,\gamma}=\overline{P_{\lambda_0,\gamma}}$, so $\nu_i(f)=0$.
\end{proof}

On the other hand, consider the functions $\{\Delta_{w_0\omega_i,\omega_i}e^{\omega_i}, \Delta_{\omega_i,w_0\omega_i}e^{\omega_i}\mid i\in I\}$ in $\BC[G\x^ZT]$. They are prime, defining the irreducible subvarieties $\overline{B^-s_iw_0 B^-}\x^Z T$ (resp., $\overline{B^+s_iw_0 B^+}\x^Z T$) in $G\x^ZT$. Denote by $\nu_{i,-}$ (resp., $\nu_{i,+}$) the corresponding valuations. By Lemma \ref{lem:PullbackCluster}, these functions are frozen variables in $\wt{\bfs}(\bfi)$ under the isomorphism $\Psi_\bfi$, so we also have the corresponding cluster valuations $\nu^\cl_{i,\pm}$. For brevity, we will denote corresponding prime functions by $A_{i,-}$ (resp. $A_{i,+}$).

The following proposition follows from the same techniques introduced in \cite{QY25}*{Appendix~B}.

\begin{prop}
\label{prop:ValuationAgreeGroup}
For any $i\in I$, the valuations $\nu^\cl_{i,\pm}$ and $\nu_{i,\pm}$ agree.
\end{prop}

\begin{proof}
The proof of $\nu^\cl_{i,\pm}(f)\leq \nu_{i,\pm}(f)$ for any cluster polynomial $f$ in $\wt{\bfs}(\bfi)$ is the same as that of Proposition~\ref{prop:ValuationAgreeVinberg}.

Conversely, let $f\in \BC[G\x^ZT]$ and we express $f$ as a Laurent polynomial in $\LP(\wt{\bfs}(\bfi))$. Let $M$ be the cluster monomial with the minimal degree such that $fM$ is a cluster polynomial. Suppose $\nu^\cl_{i,+}(f)<0$. Then $M$ is divisible by $A_{i,+}$, so $fM=0$ in the quotient $\BC[G\x^ZT]/(A_{i,+})$.

By the description of double Bruhat cells, $\BC[G^{w_0s_i,w_0}\x^ZT]$ is a localization of the quotient $\BC[G\x^ZT]/(A_{i,+})$, and the cluster variables in $\wt{\bfs}(\bfi)$ except $A_{i,+}$ correspond to cluster variables in a seed $\wt{\bfs}(\bfi)'$ for $\BC[G^{w_0s_i,w_0}\x^ZT]$. By the minimality of $M$, $fM$ is not divisible by $A_{i,+}$ as a cluster polynomial, so $fM\neq0$ in $\BC[G^{w_0s_i,w_0}\x^ZT]$, which is a contradiction. In other words, $\nu^\cl_{i,+}(f)\geq0$ and similarly $\nu^\cl_{i,-}(f)\geq0$. Since $\nu^\cl_{i,+}(A_{i,+})=\nu_{i,+}(A_{i,+})=1$ and $\nu^\cl_{i,-}(A_{i,-})=\nu_{i,-}(A_{i,-})=1$, we conclude that $\nu^\cl_{i,\pm}\geq \nu_{i,\pm}$.
\end{proof}

We are now ready to prove the main result of the paper.

\begin{theorem}
\label{thm:mainthm}
Let $\bfi$ be a double reduced word for $(w_0,w_0)$. Set $\Sigma=\wt{J}_\fro\backslash I'$. The isomorphism $\Psi_\bfi$ restricts to isomorphisms
\[
\overline{\CU}^\Sigma(\wt{\bfs}(\bfi))\overset{\sim}{\rightarrow}\BC[G\x ^Z T],\qquad \overline{\CU}(\wt{\bfs}(\bfi))\overset{\sim}{\rightarrow}\BC[\Env G].
\]
In other words, we have the following commutative diagram
\begin{equation}\label{eq:cr}
    \begin{tikzcd}
        & \CU(\wt{\bfs}(\bfi)) \arrow[r,"\sim"] & \BC[G^{w_0,w_0}\x ^Z T] \\
        & \overline{\CU}^\Sigma(\wt{\bfs}(\bfi)) \arrow[r,"\sim"] \arrow[u,hook] & \BC[G\x^Z T] \arrow[u,hook] \\
        & \overline{\CU}(\wt{\bfs}(\bfi)) \arrow[u,hook] \arrow[r,"\sim"] & \BC[\Env G] \arrow[u,hook]
    \end{tikzcd}
\end{equation}
\end{theorem}

\begin{proof}
Recall the characterization of partially compactifications using valuations in \eqref{eq:uos}. By Proposition~\ref{prop:ValuationAgreeGroup}, $\nu^\cl_{i,\pm}=\nu_{i,\pm}$, so
\[
\Psi_\bfi(\overline{\CU}^\Sigma(\wt{\bfs}(\bfi)))=\{f\in\BC[G^{w_0,w_0}\x ^Z T]\mid \nu_{i,\pm}(f)\geq 0,\forall i\in I\}=\BC[G\x^ZT].
\]

Similarly, by Proposition~\ref{prop:ValuationAgreeVinberg}, $\nu^\cl_i=\nu_i$, so
\[
\Psi_\bfi(\overline{\CU}(\wt{\bfs}(\bfi)))=\{f\in \BC[G\x^ TZ]\mid \nu_i(f)\geq0,\forall i\in I\}=\BC[\Env G].
\]
\end{proof}

Recall the frozen vertices $I'\subset \wt{J}_\fro$ of the seed $\wt{\bfs}(\bfi)$. By \eqref{eq:ff}, there is a canonical map $\pi^{I'}:\Spec \overline{\CU}(\wt{\bfs}(\bfi))\rightarrow\BC^{I'}\cong \BC^I$. The last isomorphism in \eqref{eq:cr} induces an isomorphism
\begin{equation}
\Env G\overset{\sim}{\longrightarrow} \Spec \overline{\CU}(\wt{\bfs}(\bfi))
\end{equation}
as varieties. By direct inspection of the frozen variables, we obtain the following corollary, which provides a cluster interpretation for the Vinberg degeneration $\pi:\Env G\rightarrow \BC^I$.

\begin{cor}\label{cor:spec}
The diagram
\begin{equation}
    \begin{tikzcd}
    \Env G \arrow[dr,"\pi"'] \arrow[rr,"\sim"] && \Spec \overline{\CU}(\wt{\bfs}(\bfi)) \arrow[dl,"\pi^{I'}"] \\ & \BC^I
    \end{tikzcd}
\end{equation}
commutes.
\end{cor}

\begin{example}
For $G=SL_2$, the quiver for $\wt{\bfs}(\bfi)$ is given by
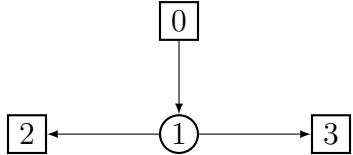
\begin{center}
\begin{tikzpicture}[every node/.style={inner sep=0, minimum size=0.5cm, thick, fill=white, draw}, x=2cm, y=1cm]
\node (1) at (1,0) [circle]{$1$};
\node (2) at (0,0) {$2$};
\node (3) at (2,0) {$3$};
\node (0) at (1,1.5) {$0$};
\drawpath{0,1,2}{black}
\drawpath{1,3}{black}
\end{tikzpicture}
\end{center}
where $1$ is the only unfrozen vertex and $\Sigma=\{2,3\}$. The cluster variables are $A_1=x_{11}e^\omega$, $A_2=x_{12}e^\omega$, $A_3=x_{21}e^\omega$, $A_1'=x_{22}e^\omega$, $A_0=e^\alpha=e^{2\omega}$.

We have
\[
\overline{\CU}^\Sigma(\wt{\bfs}(\bfi))=\BC[A_1,A_2,A_3,A_1',A_0^{\pm}]/(A_1A_1'-A_2A_3-A_0)\simeq \BC[GL_2],
\]
and
\[
\overline{\CU}(\wt{\bfs}(\bfi))=\BC[A_1,A_2,A_3,A_1',A_0]/(A_1A_1'-A_2A_3-A_0)\simeq \BC[M_2].
\]
This recovers the direct calculations in Example~\ref{eg:part1}.
\end{example}

\section{partially compactification of Levi subgroups}
\label{sec:Levi}

\subsection{partially compactification of Levi subgroups}
\label{sec:Levi:cpt}

Still let $G$ be a simply connected semisimple group. Let $\dot{G}$ be the minimal Kac--Moody group associated with $\dot{\CD}=(\dot{I},\dot{A},\dot{X},\dot{Y},(\dot{\alpha}_i)_{i\in \dot{I}},(\dot{\alpha}_i^\vee)_{i\in\dot{I}})$, such that there is a subset $I\subset \dot{I}$ giving a sub-root datum which is identical to the root datum of $G$.

Let $\dot{L}_I\subset \dot{G}$ be the standard Levi subgroup for $I$, so $\dot{L}_I$ is a reductive group with derived group $G$. $\dot{L}_I$ contains the double Bruhat cell $\dot{G}^{w_I,w_I}$. By Theorem~\ref{thm:DBCcluster}, we have an isomorphism $\BC[\dot{G}^{w_I,w_I}]\simeq \CU(\dot{\bfs})$, where $\dot{\bfs}$ is the seed associated to any double reduced word for $(w_I,w_I)$.

Let $\Sigma$ be the frozen vertices of $\dot{\bfs}$ on levels in $I$. The same statement as Proposition~\ref{prop:ValuationAgreeGroup} applied to $\dot{G}$ allows us to determine the partially compactification of $\CU(\dot{\bfs})$ along $\Sigma$.

\begin{prop}
\label{prop:LeviCluster}
The partially compactification $\overline{\CU}^\Sigma(\dot{\bfs})\simeq \BC[\dot{L}_I]$.
\end{prop}

Recall that $\wt{G}$ is the framed group associated with $G$ and $\wt{L}_I\subset \wt{G}$ is the standard Levi subgroup associated with $I$ (cf. Section \ref{sec:cluster:framed}). There is a $G\times G$-equivariant homomorphism $\rho:\dot{L}_I\to \wt{L}_I$ which is identity on $G$, and for $j\in \dot{I}\backslash I$,
\[
\dot{\alpha}_j^\vee(t)\mapsto \prod_{i\in I}\wt{\alpha}_{i'}^\vee(t)^{-\dot{a}_{ij}}.
\]
For $v,w\in W_I$ and $i\in I$, we have $\rho^*(\Delta_{v\wt{\omega}_i,w\wt{\omega}_i})=\Delta_{v\dot{\omega}_i,w\dot{\omega}_i}$ and $\rho^*(\Delta_{\wt{\omega}_{i'},\wt{\omega}_{i'}})=\prod_{j\in \dot{I}\backslash I}\Delta_{\dot{\omega}_j,\dot{\omega}_j}^{-\dot{a}_{ij}}$.

In terms of the cluster algebra, $\rho^*$ is a coefficient specialization in the sense of \cite{FZ03}. We have an isomorphism
\[
\BC[\dot{L}_I]\simeq \BC[\wt{L}_I]\ox_{\BC[\Delta_{\wt{\omega}_{i'},\wt{\omega}_{i'}}^\pm]_{i\in I}}\BC[\Delta_{\dot{\omega}_j,\dot{\omega}_j}^\pm]_{j\in \dot{I}\backslash I}.
\]

\begin{prop}
\label{prop:LeviCpt}
The partially compactification $\overline{\CU}(\dot{\bfs})$ is the coordinate ring of a flat reductive monoid with unit group $\dot{L}_I$.
\end{prop}

\begin{proof}
We first show that $\overline{\CU}(\dot{\bfs})$ is stable under the $\dot{L}_I\x \dot{L}_I$-action on $\BC[\dot{L}_I]$. We have
\[
\overline{\CU}(\dot{\bfs})=\BC[\dot{L}_I]\cap \overline{\LP}(\dot{\bfs}).
\]
Since $\overline{\LP}(\dot{\bfs})$ is stable under action by $\dot{\alpha}_j^\vee(t)$ for any $j\in \dot{I}\backslash I$, so is $\overline{\CU}(\dot{\bfs})$.

Meanwhile, since $-\dot{a}_{ij}\geq0$ for all $i\in I$, $j\in \dot{I}\backslash I$, one has
\begin{equation}
\label{eq:universalProp}
\overline{\CU}(\dot{\bfs})\simeq \BC[\Env G]\ox_{\BC[\Delta_{\wt{\omega}_{i'},\wt{\omega}_{i'}}]_{i\in I}} \BC[\Delta_{\dot{\omega}_j,\dot{\omega}_j}]_{j\in \dot{I}\backslash I}.
\end{equation}

By Theorem~\ref{thm:mainthm}, the partially compactification $\BC[\Env G]$ is stable under the $G\x G$-action on $\BC[\wt{L}_I]$. $G\x G$ also acts trivially on $\BC[\Delta_{\dot{\omega}_j,\dot{\omega}_j}]_{j\in \dot{I}\backslash I}$. Since $\rho^*$ is $G\x G$-equivariant, we conclude that $\overline{\CU}(\dot{\bfs})$ is stable under the $G\x G$-action, hence the $\dot{L}_I\x \dot{L}_I$-action by combining with the first paragraph.

By \cite{Rit98}*{Proposition~1}, we see that $\overline{\CU}(\dot{\bfs})$ is the coordinate ring of a flat reductive monoid with unit group $\dot{L}_I$. Normality of $\overline{\CU}(\dot{\bfs})$ follows from normality for general (partially compactified) upper cluster algebras, since it is the intersection of the rings $\overline{\LP}^\Sigma(\bfs)$ for $\bfs\sim \dot{\bfs}$, which are integrally closed. The flatness follows from the base change in \eqref{eq:universalProp}.
\end{proof}

\begin{rem}
The isomorphism \eqref{eq:universalProp} is exactly the isomorphism from the universal property of the Vinberg monoid. In this sense, the universality of the Vinberg monoid is partly explained by the universality of coefficient specialization in cluster algebras.
\end{rem}

\subsection{Reductive groups as Levi subgroups}
\label{sec:Levi:reductive}

Recall that the reductive groups with simply connected derived group $G$ are classified by the pairs $(D,\varphi)$ where $D$ is a diagonalizable reductive group, $\varphi$ is an inclusion $Z(G)\hookrightarrow D$, such that the cokernel is a torus \cite{Mil17}*{Remark~19.30}.

\begin{lem}
\label{lem:reductiveLevi}
Let $H$ be a reductive group with simply connected derived group $G$. Then there is a minimal Kac--Moody group $\dot{G}$ with standard Levi subgroup $\dot{L}_I$ isomorphic to $H$.
\end{lem}

\begin{proof}
Taking character groups, the data $(D,\varphi)$ is equivalent to a surjection $\varphi^*: M\to X^*(Z(G))$, where $M$ is a finitely generated abelian group, whose kernel is free. Meanwhile, $X^*(Z(G))$ is also the quotient of the weight lattice $X$ of $G$ by the root lattice $R$.

We have the following diagram
\begin{center}
\begin{tikzcd}
 & & 0\ar[d] & 0\ar[d] & \\
 & & \BZ^r\ar[d,"g\oplus h"] & R\ar[d] & \\
0\ar[r] & \BZ^k\ar[r] & \BZ^k\oplus X\ar[d,"\iota\oplus f"]\ar[r] & X\ar[d]\ar[r]\ar[dl,"f"] & 0 \\
0\ar[r] & \BZ^k\ar[r,"\iota"] & M \ar[d]\ar[r,"\varphi^*"] & X^*(Z(G))\ar[d]\ar[r] & 0 \\
 & & 0 & 0 &
\end{tikzcd}
\end{center}
Here $f$ is a lift of $X\to X^*(Z(G))$, $\iota:\BZ^k\to M$ is the kernel of $\varphi^*$. A standard exercise shows that $\iota\oplus f$ is surjective. Let $g\oplus h$ be the kernel of $\iota\oplus f$. Then $h$ is injective and $h(\BZ^r)=R$, so we may assume $h$ is the inclusion $R\hookrightarrow X$.

Any other choice of $f$ is given by $f_\delta=f-\iota\delta$ for some $\delta: X\to \BZ^k$. The corresponding $g_\delta$ is given by $g_\delta=g+\delta|_R$. In particular, we may choose $\delta$ such that each $g_\delta(\alpha_i)$ have negative coefficients in $\BZ^k$. Taking $\BZ^k\oplus X$ as the new weight lattice, we can extend the root datum for $G$ to a Kac--Moody root datum, which gives the desired minimal Kac--Moody group $\dot{G}$.

Finally, for the standard Levi subgroup $\dot{L}_I$, by construction $X^*(Z(\dot{L}_I))\simeq M$, and the inclusion $Z(G)\hookrightarrow Z(\dot{L}_I)$ induces the surjection $M\to X^*(Z(G))$, so $\dot{L}_I\simeq H$.
\end{proof}

Combining with Propositions \ref{prop:LeviCluster} and \ref{prop:LeviCpt}, we obtain the following theorem.

\begin{theorem}
\label{thm:reductiveMonoidCluster}
Let $H$ be a connected reductive group with simply connected derived group. Then the coordinate ring of $H$ is a partially compactified upper cluster algebra $\overline{\mathcal{U}}^\Sigma$, whose further partially compactification $\overline{\mathcal{U}}$ along all frozen variables is the coordinate ring of a flat reductive monoid with unit group $H$.
\end{theorem}

\begin{rem}
As seen from the proof of Lemma~\ref{lem:reductiveLevi}, there are many choices of $\dot{G}$ containing $H$ as a standard Levi subgroup. This gives different cluster structures on $H$ which are quasi-equivalent in the sense of \cite{Fra16}. We also get different reductive monoids with unit group $H$. However, this construction does not cover all such reductive monoids, as the abelianizations are always isomorphic to an affine space $\BC^n$.
\end{rem}

\begin{example}
We continue with our running example $G=SL_2$. The only rank 2 reductive groups with derived group $SL_2$ are $GL_2$ and $SL_2\x \BC^\x$.

For any non-negative integer $k$, let $\dot{G}$ be the minimal Kac--Moody group associated with the generalized Cartan matrix $\begin{pmatrix}
2 & -1-2k\\
-1-2k & 2
\end{pmatrix}$.

We have
\[
\BC[\dot{L}_I]\simeq \BC[A_1,A_2,A_3,A_1',A_0^\pm]/(A_1A_1'-A_2A_3-A_0^{1+2k}),
\]
and $\dot{L}_I$ is isomorphic to $GL_2$ via the isomorphism $\BC[\dot{L}_I]\to \BC[GL_2]$,
\[
(A_1,A_2,A_3,A_1',A_0)\mapsto(x_{11}\Delta^k,x_{12}\Delta^k,x_{21}\Delta^k,x_{22}\Delta^k,\Delta),
\]
where $\Delta=x_{11}x_{22}-x_{12}x_{21}$.

Under this isomorphism, the partially compactification $\overline{\CU}(\dot{\bfs})$ of $\BC[\dot{L}_I]$ is the subalgebra of $\BC[x_{11},x_{12},x_{21},x_{22}]$ generated by $x_{11}\Delta^k$, $x_{12}\Delta^k$, $x_{21}\Delta^k$, $x_{22}\Delta^k$, and $\Delta$. We see that $\overline{\CU}(\dot{\bfs})$ is stable under the $GL_2\x GL_2$ action. Let $M_k=\text{Spec }\overline{\CU}(\mathbf{s})$. Then $M_k$ is a reductive monoid with unit group $GL_2$.

Let $T$ be the standard torus of $GL_2$. By \cite{BK05}*{Theorem~6.2.13}, the normal reductive monoids $M$ with unit group $GL_2$ can be classified by the $W$-stable rational polyhedral cone with non-empty interior, by identifying $M$ with the convex cone in $X^*(T)_\BR=\BR\otimes_\BZ X^*(T)$ generated by weights of $T$ in $\BC[\overline{T}]$. In our example, let $\varepsilon_i:\begin{pmatrix}
    x_{11} & 0 \\ 0 & x_{22}
\end{pmatrix}\mapsto x_{ii}$ for $i=1,2$ be the standard weights of $T$. Then   
\[
\BC[\overline{T}]\simeq \BC[x_{11}^{k+1}x_{22}^k,x_{11}^kx_{22}^{k+1},x_{11}x_{22}],
\]
and the convex cone is given by the vectors $(k+1)\varepsilon_1+k\varepsilon_2$ and $k\varepsilon_1+(k+1)\varepsilon_2$. By \cite{Vin95}*{Example~2}, these are exactly all the flat reductive monoids with unit group $GL_2$ with the abelianlization isomorphic to the affine line. 

On the other hand, for the generalized Cartan matrix $\begin{pmatrix}
2 & -2k\\
-2k & 2
\end{pmatrix}$, we get the reductive group $SL_2\x \BC^\x$ instead.
\end{example}

\subsection{Reductive monoids}
\label{sec:Levi:monoid}

Let $M$ be a flat reductive monoid.

The abelization $A(M)$ is again a monoid, and its unit group is isomorphic to the torus $G(M)/G_M$. In particular, if $A(M)$ is isomorphic to an affine space $\BC^k$, we have $A(M)\simeq \BC^k$ as algebraic monoids, where the monoid structure on $\BC^k$ is coordinate-wise multiplication. In this subsection, $\BC^k$ will always be equipped with this monoid structure.

\begin{theorem}\label{thm:affineAbelianCluster}
Let $M$ be a flat reductive monoid. Suppose that the semisimple part of $M$ is simply connected and the abelianization of $M$ isomorphic to an affine space $\mathbb{C}^n$. Then the coordinate ring $\BC[M]$ is a partially compactified upper cluster algebra with all frozen variables not invertible.
\end{theorem}

We first need a lemma.

\begin{lem}
\label{lem:affineSpaceHom}
Let $f:\BC[x_1,\cdots,x_n]\to \BC[y_1,\cdots,y_m]$ be a ring homomorphism such that the induced morphism $f^*:\BC^m\to \BC^n$ is a monoid homomorphism. Then $f$ maps each $x_i$ to a monomial in the $y_j$.
\end{lem}

\begin{proof}
Let
\[
\Delta: \BC[y_1,\cdots,y_m]\to \BC[y_1,\cdots,y_m]\ox \BC[y_1,\cdots,y_m]
\]
be the comultiplication of $\BC[y_1,\cdots,y_m]$. The ring $\BC[y_1,\cdots,y_m]\ox \BC[y_1,\cdots,y_m]$ is $\BZ^m\oplus \BZ^m$-graded, and the image of $\Delta$ is a graded subring with grading in the diagonal $\Delta(\BZ^m)$. Since $f^*$ is a monoid homomorphism, we have $\Delta(f(x_i))=f(x_i)\ox f(x_i)$ for each $x_i$. This is only possible if $f(x_i)$ homogeneous, i.e. a monomial in the $y_j$.
\end{proof}

\begin{proof}[Proof of Theorem~\ref{thm:affineAbelianCluster}]
Let $G=G_M$ and $\CD=({I},{A},{X},{Y},({\alpha}_i)_{i\in {I}},({\alpha}_i^\vee)_{i\in{I}})$ be the root datum associated with $G$. Let $A_0=\text{Spec }\BC[e^{\alpha_i}]_{i\in I}$ be the abelianization of $\Env G$. By the universality of the Vinberg monoid, there is an algebra homomorphism $f:\BC[A_0]\rightarrow \BC[A(M)]$, such that $\BC[M]\cong \BC[\Env G]\otimes _{\BC[A_0]}\BC[A(M)]$.

Let us fix an isomorphism
\[
A(M)\simeq \BC^k=\Spec \BC[x_1,\cdots,x_k]
\]
as algebraic monoids. Since $f$ is induced from a homomorphism of algebraic monoids, by Lemma~\ref{lem:affineSpaceHom}, $f(e^{\alpha_i}) =\prod_{j=1}^k x_j^{f_{ij}}$ for some $f_{ij}\in\BN$.

Consider a generalized Cartan matrix $\dot{A}=(\dot{a}_{ij})_{i,j\in\dot{I}}$. The index set $\dot{I}=I\sqcup J$ where $J=\{1',2'\dots,k'\}$. For $i,j\in I$, set $\dot{a}_{ij}=a_{ij}$. For $i,j\in J$, set $\dot{a}_{ij}=2\delta_{ij}$. For $i\in I, j'\in J$, set $\dot{a}_{ij'}=-f_{ij}$. The symmetrizer $(\dot{d}_i)_{i\in\dot{I}}$ is defined such that $\dot{d}_i=d_i$ for $i\in I$ and $\dot{d}_j=1$ for $j\in J$. Let $\dot{\mathcal{D}}$ be the simply-connected root datum associated with the Cartan matrix $\dot{A}$. Let $\dot{W}$ be the associated Weyl group. It contains a parabolic subgroup $W_I\subset \dot{W}$ corresponding to the subset $I\subseteq \dot{I}$. Let $w_I$ be the longest element in $W_I$. Let $\dot{\mathbf{s}}$ be a seed associated with a double reduced expression of $(w_I,w_I)$ in $W\times W$. Then by the proof of Proposition \ref{prop:LeviCpt} (cf. \eqref{eq:universalProp}), we conclude that
\[
\overline{\CU}(\dot{\bfs})\simeq \BC[\Env G] \ox_{A_0}\BC[A(M)]\cong \BC[M].  
\]
The proof is completed.
\end{proof}

\begin{example}
Let $M=M_n$ be the monoid of $n\times n$ matrices, so $G(M)=GL_n$ and $G_M=SL_n$. In this case, the subalgebra $\BC[M]^{G_M\times G_M}$ of invariant functions is isomorphic to the polynomial algebra in one variable, generated by the determinant. Hence the abelianization $A(M)$ is isomorphic to the affine line. Theorem~\ref{thm:affineAbelianCluster} gives a partially compactified upper cluster algebra structure on $\BC[M_n]$, which coincides with the cluster structure in \cite{FWZ20}*{Theorem 6.6.1}.
\end{example}

\end{document}